\documentclass{article}%
\usepackage{graphicx}
\usepackage{amsmath,amssymb,amsfonts}%
\usepackage{theorem}%
\usepackage{color}%
\usepackage{hyperref}
\usepackage[font={small, it}]{caption}
\usepackage{caption}
\usepackage{subcaption}
\usepackage{tikz}

\theoremstyle{change}%
\newtheorem{definition}{Definition:}[section]%
\newtheorem{proposition}[definition]{Proposition:}%
\newtheorem{theorem}[definition]{Theorem:}%
\newtheorem{lemma}[definition]{Lemma:}%
{\theorembodyfont{\rmfamily}\newtheorem{remark}[definition]{Remark:}}%
{\theorembodyfont{\rmfamily}\newtheorem{example}[definition]{Example:}}%

\newenvironment{proof}
{{\bf Proof:}}
{\qquad \hspace*{\fill} $\Box$}%

\newcommand{\fg}{\mathfrak{g}}%
\newcommand{\tr}{\operatorname{tr}}%
\newcommand{\id}{\operatorname{id}}
\newcommand{\rme}{\mathrm{e}}%

\newcommand{\CC}{\mathcal{C}}%
\newcommand{\ZC}{\mathcal{Z}}%
\newcommand{\XC}{\mathcal{X}}%
\newcommand{\DC}{\mathcal{D}}%
\newcommand{\R}{\mathbb{R}}%
\newcommand{\Z}{\mathbb{Z}}%

\begin{document}

	\title{Almost-Riemannian Structures on nonnilpotent, solvable 3D Lie groups}%
	\author{V\'{\i}ctor Ayala%
		\thanks{
			Supported by Proyecto Fondecyt n$%
			{{}^\circ}%
			$ 1190142. Conicyt, Chile.} \\
		Instituto de Alta Investigaci\'{o}n\\
		Universidad de Tarapac\'{a}, Arica, Chile \and
Adriano Da Silva\thanks{ Supported by Fapesp grant $n^{o}$ 2020/12971-4 and CNPq grant n${{}^o}$ 309820/2019-7}\\
		Departamento de Matem\'atica,\\Universidad de Tarapac\'a - Iquique, Chile.
		\and
 Danilo A. Garc\'{\i}a Hern\'andez\thanks{ Supported by Capes Finance Code 001}\\
		Instituto de Matem\'{a}tica\\
		Universidade Estadual de Campinas, Brazil\\
	}
	\date{\today }
	\maketitle
	
	\begin{abstract}
		 In this paper we study Almost-Riemannian Structures (ARS) on the class of nonnilpotent, solvable, conneted 3D Lie groups. The nice structures present in such groups allow us to show that the singular locus of ARSs on such groups are always embedded submanifolds.
	\end{abstract}
	
	{\small {\bf Keywords:} Almost-Riemannian geometry, solvable Lie groups} 
	
	{\small {\bf Mathematics Subject Classification (2020): 22E15, 22E25, 53C17, 53C15.}}%
	
	\section{Introduction}
	
	Almost-Riemannian Structures (ARS) appears in the literature as a part of sub-Riemannian geometry. Roughly speaking, an ARS on a differentiable manifold can be seen, at least locally, as a smooth orthonormal frame that degenerates on a singular set, called the singular locus. They appear naturally when studying limits of Riemmanian metrics or even in the study of certain class of hypoelliptic operators \cite{Gru}. Such concept appeared first in the work \cite{tak} but has aroused interest more recently as shown in the works \cite{Ag0, Ag1, Ag2, Bon1, Bon2, Boscain1, Boscain2, Boscain3, Boscain4, Boscain5}.  
	
    On a connected $n$-dimensional Lie group $G$, an ARS is naturally obtained by considering $n-1$ left-invariant vector fields and one linear vector field. These family of vector fields are assumed to have full rank at a nonempty subset of $G$ and to satisfy the Lie algebra rank condition. The singular locus is the set of points were these vector fields fail to be linearly independent. It is an analytic set, but in general is neither a Lie subgroup nor a submanifold as showed in \cite{PJAyGZ1}. 
	
	In \cite{Boscain5}, the authors study the Laplace-Beltrami operator associated with an ARS. They conclude that in some cases, the singular locus of an ARS works as a barrier for the wave or the heat equations, adding more importance for the study of such subset. In this paper we study the singular locus of ARSs on connected 3D solvable nonnilpotent Lie groups. By general results presented in an Appendix, our task is reduced to the case where the group under consideration is a simply connected, connected group. In order to study the singular locus, we start by describing the group of automorphisms at the algebra and group level of the 3D groups under consideration. Such descriptions allows us to simplify our choices of more ``appropriated'' ARSs. Under these simplification we can show that the singular locus of simple ARSs on these groups are always embedded submanifolds. As a consequence, we are able to obtain some results on the connectedness of the singular locus and also to analyze how exponential curves and the flow of a linear vector field crosses the singular locus. These last results goes in the opposite direction of what happens with the solutions of the wave or heat equations.
    
    The paper is divided as follows: Section 2 is used to set some preliminaries needed on the paper. We define here the class of 3D groups and algebras we are interested and obtain algebraic expressions for the algebra of derivations and the group of automorphisms at algebra and group level. In Section 3 we define simple ARSs on Lie groups. We show here that automorphisms can be seen as isometries. This result allows us to make some assumptions on our ARSs in order to simplify it.  In section 4 we prove our main results. By using the nice structures present on our class of 3D Lie groups, we are able to show that the singular locus of simple ARS on these groups are always submanifolds. This results is not true for all 3D Lie groups as showed in \cite[Section 3.5.3]{PJAyGZ1}. A sufficient condition for the connectedness of the singular locus is also obtained. We finish the section with an analysis of how the exponential curves crosses the singular locus. We finish the paper with two Appendices, the first one containing a result about the asymptotic behavior of $2\times 2$ matrices and the second one showing that simple ARS's on connected Lie groups can be lifted to its simply connected lift. As a consequence, all the results contained in Sections  4 are true for any 3D connected nonnilpotent solvable Lie groups.

	\section{Preliminaries}
	
	In this section we discuss 3D Lie groups and algebras. Our aim here is to introduce and prove some preliminares results concerning automorphisms of algebra and group, invariant and linear vector fields.

	\subsection{Solvable nonnilpotent 3D Lie algebras}
	
	By \cite[Chapter 7]{Oni} any solvable, nonnilpotent 3D Lie algebra is isomorphic to a semi-direct product of the form $\fg(\theta)=\R\times_{\theta}\R^2$ with bracket 
	determined by the relation  
	$$[(a, 0), (0, w)]=(0, a\theta w), \;\;\;(a, w)\in \R\times\R^2,$$
	and $\theta$ a $2\times 2$ matrix of one of the following types:
	$$\left(\begin{array}{cc}
	1  & 1\\ 0 & 1 
	\end{array}\right), \;\;\;\;\left(\begin{array}{cc}
	1  & 0\\ 0 & \lambda 
	\end{array}\right), \lambda\in[-1, 1]\;\;\;\;\mbox{ or }\;\;\;\;\left(\begin{array}{cc}
	\lambda  & -1\\ 1 & \lambda 
	\end{array}\right), \lambda\in\R.$$
	
	Let us remember that an automorphism and a derivation of a Lie algebra $\fg$ are, respectively, linear maps $\phi\in\mathrm{Gl}(\fg)$ and $\DC\in\mathfrak{gl}(\fg)$ satisfying
	$$\forall X, Y\in\fg, \;\;\; \phi[X, Y]=[\phi X, \phi Y]\;\;\;\mbox{ and }\;\;\;\DC[X, Y]=[\DC X, Y]+[X, \DC Y].$$
	Let $M\in\mathfrak{gl}(\fg(\theta))$ and assume that $M\R^2\subset\R^2$. Then, $M$ can be written as 
	$$M=\left(\begin{array}{cc}
	\varepsilon & 0\\ \eta & P
	\end{array}\right), \eta \in\R^2, \varepsilon\in\R \;\mbox{ and }\;P\in\mathfrak{gl}(2, \R).$$
	
	Since $\{0\}\times\R^2$ is the nilradical of $\fg(\theta)$, it is invariant by automorphism and derivations. Hence, if $\phi\in\mathrm{Aut}(\fg(\theta))$ the relation  
	$$\phi[(a, 0), (0, v)]=[\phi(a, 0), \phi(0, v)], \;\;\forall (a, v)\in\R\times\R^2,$$
	implies that
	$$\phi=\left(\begin{array}{cc}
	\varepsilon & 0\\ \eta & P
	\end{array}\right)\in\mathrm{Aut}(\fg(\theta))\;\;\mbox{ if and only if }\;\; P\theta=\varepsilon \theta P.$$
	
	Similarly, if $\DC\in\mathrm{Der}(\mathfrak{g}(\theta))$ the relation 
	$$\DC[(a, 0), (0, v)]=[\DC(a, 0), (0, v)]+[(a, 0), \DC(0, v)],  \;\;\forall (a, v)\in\R\times\R^2$$
	gives us that 
	we get that 
	$$\DC=\left(\begin{array}{cc}
	\varepsilon & 0 \\ \eta & A
	\end{array}\right)\in\mathrm{Der}(\mathfrak{g})\;\;\;\iff\;\;\; A\theta-\theta A=\varepsilon\theta.$$
	The previous calculations imply the following:	
	
	\begin{proposition}
		\label{autoLie}
		It holds that:
		$$\mathrm{Der}(\fg(\theta))=\left\{\left(\begin{array}{cc}
		0 & 0\\ \xi & A
		\end{array}\right), \xi\in\R^2, A\in\mathfrak{gl}(2, \R), \mbox{ with } A\theta=\theta A\right\},$$
		and 

		$$\mathrm{Aut}(\fg(\theta))=\left\{\left(\begin{array}{cc}
		\varepsilon & 0\\ \eta & P
		\end{array}\right), \eta\in\R^2, P\in\mathrm{Gl}(\R^2), \mbox{ with } P\theta=\varepsilon\theta P\right\},$$
		where $\varepsilon=1$ when $\tr(\theta)\neq0$ or $\varepsilon\in\{-1, 1\}$ if $\tr\theta=0$.
		
	\end{proposition} 
	
	\begin{proof}
		By our previous calculations,
		$$\phi=\left(\begin{array}{cc}
		\varepsilon & 0\\ \eta & P
		\end{array}\right)\in\mathrm{Aut}(\fg(\theta))\;\;\mbox{ if and only if }\;\;P\theta=\varepsilon \theta P.$$
		Moreover, the fact that $\phi$ is invertible implies that $\varepsilon\det P=\det \phi\neq0$ and hence,  
		$$P\theta =\varepsilon \theta P\;\;\;\iff\;\;\;\varepsilon\theta=P\theta P^{-1}.$$
		Assuming $\tr \theta\neq 0$ gives us that
		$$\varepsilon\tr(\theta)=\tr(P\theta P^{-1})=\tr(\theta)\;\;\implies\;\;\varepsilon=1\;\;\mbox{ and }\;\;P\theta =\theta P.$$
		On the other hand, $\tr(\theta)=0$ implies necessarily
		$$\theta=\left(\begin{array}{cc}
		1 & 0 \\ 0 & -1
		\end{array}\right)\;\;\;\mbox{ or }\;\;\;\theta=\left(\begin{array}{cc}
		0 & -1 \\ 1 & 0
		\end{array}\right),$$
		and in both cases $\det(\theta)\neq 0$. Therefore, 
		$$\varepsilon=1\;\;\mbox{ and }\;\;P\theta =\theta P\;\;\mbox{ or }\;\;\varepsilon=-1\;\;\mbox{ and }\;\;P\theta=-\theta P.$$
		For a derivation $\DC\in\mathrm{Der}(\fg(\theta))$ we have that 
		$$\DC=\left(\begin{array}{cc}
		\varepsilon & 0\\ \xi & A
		\end{array}\right)\in\mathrm{Der}(\fg(\theta))\;\;\mbox{ if and only if }\;\;A\theta-\theta A=\varepsilon \theta.$$
		Hence, $\tr\theta\neq0$ gives us 
		$$\varepsilon\tr\theta=\tr(A\theta-\theta A)=0\;\;\implies\;\;\varepsilon=0.$$
		On the other hand, if $\tr\theta=0$ then necessarily $\det\theta\neq 0$ and
		$$A\theta-\theta A=\varepsilon\theta\;\;\implies\;\; \varepsilon\id_{\R^2}=A-\theta A\theta^{-1}\;\;\implies\;\;2\varepsilon=\tr(A-\theta A\theta^{-1})=0,$$
		showing that, in any case, $\varepsilon=0$ and $A\theta=\theta A$, concluding the proof.
	\end{proof}
	
	\subsection{Solvable nonnilpotent simply connected connected 3D Lie groups}
	
    Up to isomorphisms, the connected, simply connected Lie group associated with the 3D Lie algebra $\fg(\theta)$ is the semi-direct product $G(\theta)=\R\times_{\rho}\R^2$, where $\rho_t=\rme^{t\theta}$ and the product satisfies
    $$(t_1, v_1)(t_2, v_2)=(t_1+t_2, v_1+\rho_{t_1}v_2), \;\;\; (t_1, v_1), (t_2, v_2)\in\R\times_{\rho}\R^2.$$
    
    As for the algebras $\fg(\theta)$, the group of automorphisms of $G(\theta)$ can be expressed in terms of $3\times 3$ matrices. In order to do that we introduce the following operator: Let $A$ be a $2\times 2$ and define  
    $$\Lambda^A:\R\times \R^2\rightarrow\R^2, \;\;\;(t, w)\mapsto \Lambda_t^Aw:=\int_0^t\rme^{sA}w \,ds.$$
    The operator $\Lambda^A$ is well defined and, for any $t, s\in\R$, it satisfies
      
    $$\hspace{-2.5cm}1.\;\Lambda^A_0=0;\hspace{3cm} 2.\; \frac{d}{dt}\Lambda^A_t=\rme^{tA}\hspace{2.2cm} 3.\; \Lambda^A_{t+s}=\Lambda^A_t+\rme^{tA}\Lambda^A_s$$
         $$4.\;\rme^{tA}-A\Lambda^A_t=\id_{\R^2}\hspace{1.5cm} 5.\;\rme^{sA}\Lambda^A_t=\Lambda^A_t\rme^{sA}\hspace{1.5cm} 6.\; \Lambda_t^A=(\rme^{tA}-\id_{\R^2})A^{-1} \mbox{ if }\det A\neq 0$$
         $$7.\;A=\left(\begin{array}{cc}
         \lambda & 0 \\ 0 & 0
         \end{array}\right), \;\lambda\neq0\;\;\;\implies\;\;\;\Lambda_s^A=\left(\begin{array}{cc}
         \frac{1}{\lambda}(\rme^{s\lambda}-1) & 0 \\ 0 & s
         \end{array}\right).$$

    Using the previous operator and Proposition \ref{autoLie} we have the following:    
   	
   	\begin{proposition}
   		\label{autoG}
	   	For the 3D Lie group $G(\theta)$, it holds that:
   		$$\mathrm{Aut}(G(\theta))=\left\{\phi(t, v)=\left(\varepsilon t, Pv+\varepsilon\Lambda_{\varepsilon t}^{\theta}\eta\right), \eta\in\R^2, P\in\mathrm{Gl}(\R^2), \mbox{ with } P\theta=\varepsilon\theta P\right\},$$
   		where $\varepsilon=1$ if $\tr\theta\neq 0$ or $\varepsilon\in\{-1, 1\}$ if $\tr\theta=0$. 
   	\end{proposition}
   	
   	\begin{proof}
   		If $\phi\in\mathrm{Aut}(G(\theta))$ we have that $(d\phi)_{(0, 0)}\in\mathrm{Aut}(\fg(\theta))$. Therefore, we can write 
   		$$(d\phi)_{(0, 0)}=\left(\begin{array}{cc}
   		\varepsilon & 0 \\ \eta & P
   		\end{array}\right), \;\;\mbox{ with }\;\; P\theta=\varepsilon \theta P\;\;\mbox{ and }\;\;\varepsilon^2=1.$$
   		The map $\psi(t, v)=\left(\varepsilon t, Pv+\varepsilon\Lambda^{\theta}_{\varepsilon t}\eta\right)$ satifies $(d\psi)_{(0, 0)}=(d\phi)_{(0, 0)}$ and, since $G(\theta)$ is connected, if we show that $\psi\in\mathrm{Aut}(G(\theta))$, the equality $\psi=\phi$ follows by general results.
   		For any $(t_1, v_1), (t_2, v_2)\in G(\theta)$ it holds that
   		$$\psi((t_1, v_1)(t_2, v_2))=\psi(t_1+t_2, v_1+\rho_{t_1}v_2)=\left(\varepsilon(t_1+t_2), P(v_1+\rho_{t_1}v_2)+\varepsilon\Lambda^{\theta}_{\varepsilon(t_1+t_2)}\eta\right)$$
   		$$=\left(\varepsilon t_1+\varepsilon t_2, Pv_1+\rho_{\varepsilon t_1}Pv_2+\varepsilon\Lambda^{\theta}_{\varepsilon t_1}\eta+\varepsilon\rho_{\varepsilon t_1}\Lambda^{\theta}_{t_2}\eta\right)=
   		\left(\varepsilon t_1+\varepsilon t_2, Pv_1+\varepsilon\Lambda^{\theta}_{\varepsilon t_1}\eta+\rho_{\varepsilon t_1}\left(Pv_2+\varepsilon\Lambda^{\theta}_{\varepsilon t_2}\eta\right)\right)$$
   		$$=(\varepsilon t_1, Pv_1+\varepsilon\Lambda^{\theta}_{\varepsilon t_1}\eta)(\varepsilon t_2, Pv_2+\varepsilon\Lambda^{\theta}_{\varepsilon t_2}\eta)=\psi(t_1, v_1)\psi(t_2, v_2),$$
   		where for the second equality we used that $P\theta=\varepsilon\theta P$ and for the third equality a property of $\Lambda^{\theta}_t$. Therefore, $\psi\in\mathrm{Aut}(G(\theta))$ and consequently $\psi=\phi$ as stated. 			
   	\end{proof}
   	
   	\begin{definition}
   		\label{linear}
    	A vector field $\XC$ on $G(\theta)$ is said to be {\bf linear} if it is complete and its flow $\{\varphi_s\}_{s\in\R}$ is a $1$-parameter subgroup of $\mathrm{Aut}(G(\theta))$.
     	\end{definition}

   	Since by differentiation, $\{(d\varphi_s)_{(0, 0)}\}_{s\in\R}$ is a $1$-parameter subgroup of $\mathrm{Aut}(\fg(\theta))$, there exists a derivation $\DC\in\mathrm{Der}(\fg(\theta))$ such that 
   	$$\forall s\in\R.\;\;\;\;(d\varphi_s)_{(0, 0)}=\rme^{s\DC}.$$
   	However, 
   	$$\DC=\left(\begin{array}{cc}
   	0 & 0 \\ \xi & A
   	\end{array}\right)\;\;\;\implies\;\;\;\rme^{s\DC}=\left(\begin{array}{cc}
   	1 & 0 \\ \Lambda^A_s\xi & \rme^{sA}
   	\end{array}\right),$$
   	and hence, by the previous calculations, we get that  
   	$$\varphi_s(t, v)=\left(t, \rme^{sA}v+\Lambda_t^{\theta}\Lambda_s^A\xi\right), \;\;\xi\in\R^2, A\in\mathfrak{gl}(2, \R), \mbox{ with }  A\theta=\theta A.$$
   	Also, the linear vector field $\XC$ can be written as
   	$$\XC(t, v)=(0, Av+\Lambda^{\theta}_t\xi), \;\;\;\xi\in\R^2, A\in\mathfrak{gl}(2, \R), \mbox{ with }  A\theta=\theta A.$$
   	By the fact that the vector $\xi\in\R^2$ and the matrix $A\in\mathfrak{gl}(2, \R)$, together with the previous properties, determine $\XC$, we will usually write $\XC=(\xi, A)$ to denote the linear vector field $\XC$.
   	
   	We denote by $\ZC_{\XC}$ the set of the {\bf singularities} of a linear vector field $\XC$, that is,  
   	$$\ZC_{\XC}=\{(t, v)\in G(\theta); \;\;\XC(t, v)=0\}.$$
   	It is well know that $\ZC_{\XC}$ is a closed subgroup with Lie algebra given by $\ker\DC$.
   	   
     \bigskip 
   	  
 	We finish this section with the explicity expression for the exponential map of a 3D Lie group and the differential of the left and right translations which will be needed ahead. The details in how to obtain such formulas can be found in \cite[Section 3]{DSAy1}.
	$$
	\exp (a, w)=\left\{\begin{array}{cc}
	(0, w)  & \mbox{ if }a=0\\
	\left(a, \frac{1}{a}\Lambda^{\theta}_aw\right), & \mbox{ if }a\neq 0		
	\end{array}\right..
	$$
	Also, for any $(t_i, v_i)\in G(\theta)$, $i=1, 2$ and $(a, w)\in\fg(\theta)$ the left and right translations satisfies  
	$$(dL_{(t_1, v_1)})_{(t_2, v_2)}(a, w)=(a, \rho_{t_1}w)\;\;\mbox{ and }\;\;(dR_{(t_1, v_1)})_{(t_2, v_2)}(a, w)=(a, w+a\theta\rho_{t_2}v_1),$$
	implying that the left and right invariant vector field associated with $Y=(a, w)\in\fg(\theta)$ are given, respectively, by
	$$Y^L(t, v)=(a, \rho_tw)\;\;\;\;\mbox{ and }\;\;\;\;Y^R(t, v)=(a, w+a\theta v).$$


		\section{Simple ARS's on $G(\theta)$}
		
	   In this section we define simple ARS's on the connected simply connected 3D Lie groups $G(\theta)$ associated with the Lie algebras $\fg(\theta)$.

		\subsection{Invariant distributions and ARS's on $G(\theta)$}

		A 2D left-invariant distribution on $G(\theta)$ is the map $\Delta^L:G(\theta)\rightarrow TG(\theta)$ given by 
		$$\Delta^L(t, v)=(dL_{(t, v)})_{(0, 0)}\Delta,$$
		where $\Delta\subset\mathfrak{g}(\theta)$ is a 2D vector subspace. We can endow $\Delta^L$ with a left-invariant Euclidean metric by considering on $\Delta$ an inner product $\langle\cdot, \cdot\rangle$ and defining 
		$$\forall X, Y\in \Delta^L(t, v), \;\;\; \langle X, Y\rangle_{(t, v)}:=\langle (dL_{(t, v)^{-1}})_{(t, v)}X, (dL_{(t, v)^{-1}})_{(t, v)}Y\rangle.$$
			Let $\XC$ be a linear vector field and consider the family of vector fields $\Sigma=\{\XC, \Delta^L\}$. We say that $\Sigma$ satisfies the {\bf Lie algebra rank condition (LARC)} if at least one of the following conditions is satisfied: 
			\begin{itemize}
				\item[(i)] $\Delta$ is not a subalgebra of $\fg(\theta)$;
				\item[(ii)] $\Delta$ is a subalgebra of $\fg(\theta)$ and $\DC\Delta\not\subset\Delta$,
			\end{itemize}
			where $\DC$ is the derivation associated with $\XC$.
			
			\begin{remark}
			Note that the family $\Sigma=\{\XC, \Delta^L\}$, with $\Delta=\{0\}\times\R^2$, does not satisfies the LARC. In fact, in this case $\Delta$ is a subalgebra and 
			$$\DC(0, w)=\left(\begin{array}{cc}
			0 & 0 \\ \xi & A
			\end{array}\right)\left(\begin{array}{c}
			0 \\ w
			\end{array}\right)=\left(\begin{array}{c}
			0 \\ Aw
			\end{array}\right)=(0, Aw)\;\;\implies\;\;\DC\Delta\subset\Delta.$$
			
			Therefore, if $\Sigma=\{\XC, \Delta^L\}$ satisfies the LARC the subspace $\Delta$ admits a basis  
			$$\{ (\sigma_1, u_1), (\sigma_2, u_2)\}, \;\;\;\mbox{ with }\;\;\;\sigma_1^2+\sigma_2^2\neq0.$$
			Moreover, the intersection $\Delta\cap(\{0\}\times\R^2)$ is one-dimensional. Let us denote by $l_{\Delta}$ the line in $\R^2$ satisfying 
			$$\{0\}\times l_{\Delta}=\Delta\cap(\{0\}\times\R^2).$$
           \end{remark}

		\begin{definition}
			\label{ARS}
			A simple ARS $\Sigma$ on the 3D Lie group $G(\theta)$ is defined by the family $\Sigma=\{\XC, \Delta^{L}\}$, where $\XC$ is a linear vector field and $\Delta^{L}$ is a 2D left-invariant distribution $\Delta$ endowed with a left-invariant Euclidean metric satisfying  
			\begin{itemize}
				\item[(i)] The set $\{(t, v)\in G(\theta);\;\;\XC(t, v)\notin\Delta^L(t, v)\}$ is nonempty;
				\item[(ii)] $\Sigma$ satisfies the LARC.
			\end{itemize}
			The metric on $G(\theta)$ associated with $\Sigma$ is defined by declaring that $\XC$ is unitary and orthogonal to $\Delta^L$ at every point.
		\end{definition}

		The {\bf singular points} $\ZC$ is the set of points $(t, v)\in G(\theta)$ where the previous metric fails to be Riemannian, that is
		$$\ZC:=\{(t, v)\in G(\theta); \;\XC(t, v)\in \Delta^L(t, v)\}.$$
		In particular, the set of singularities $\ZC_{\XC}$ of $\XC$ is contained in $\ZC$.

		The {\bf almost-Riemannian norm} defined by $\Sigma$ on $T_{(t, v)}G(\theta)$ is defined by
		$$\|X\|_{\Sigma, (t, v)}=\min\left\{\sqrt{\alpha_0^2+\alpha_1^2+\alpha_2^2}; \;\alpha_0\XC(t, v)+\alpha_1Y_1^L(t, v)+\alpha_2Y^L_2(t, v)=X\right\},$$
		where $\{Y_1, Y_2\}$ is an orthonormal basis of $\Delta$. It holds that $\|X\|_{\Sigma, (t, v)}=\infty$ when $(t, v)\in\ZC$ and $X\notin\Delta(t, v)$. 
		
		A diffeomorphism $\psi:G(\theta)\rightarrow G(\theta)$ between two ARS's $\Sigma_1$ and $\Sigma_2$ on a 3D Lie group $G(\theta)$ is an {\bf isometry} if   
		$$\forall (t, v)\in G(\theta), Z\in T_{(t, v)}G(\theta)\hspace{1.5cm}\|(d\psi)_{(t, v)}Z\|_{\Sigma_2, \phi(t, v)}=\|Z\|_{\Sigma_1, (t, v)}.$$

  	The next result shows that elements in $\mathrm{Aut}(G(\theta))$ can be seen as isometries between ARSs. This will help us to simplify, when necessary, the ARS under consideration.
		
		\begin{proposition}
			\label{difeo}
			Let $\Sigma=\{\XC=(\xi, A), \Delta^L\}$ be an ARS on $G(\theta)$ and 
			$$\psi(t, v)=(\varepsilon t, Pv+\varepsilon\Lambda_{\varepsilon t}^{\theta}\eta)\;\;\mbox{ an automophism of }G(\theta).$$
			The family 
			$$\Sigma_{\psi}=\left\{\XC_{\psi}=\left( P^{-1}(\varepsilon\xi+ A\eta), P^{-1}AP\right), \Delta^L_{\psi}:=(\psi_*)^{-1}\Delta^L\right\},$$
			is an ARS and $\psi$ is an isometry between $\Sigma_{\psi}$ and $\Sigma$, where the left-invariant metric on $\Delta_{\psi}$ is the one that makes $(d\phi)_{(0, 0)}\bigl|_{\Delta_{\phi}}$ an isometry.
		\end{proposition}
		
		\begin{proof}
			By definition, we have that  
			$$(d\psi)_{(t, v)}\XC_{\psi}(t, v)=\left(\begin{array}{cc}
			\varepsilon & 0 \\ \eta & P
			\end{array}\right)\left(\begin{array}{c}
			0 \\ P^{-1}APv+\Lambda_t^{\theta}P^{-1}(\varepsilon\xi+A\eta)
			\end{array}\right)=\left(\begin{array}{c}
			0 \\ APv+P\Lambda_t^{\theta}P^{-1}(\varepsilon\xi+ A\eta)
			\end{array}\right)$$
			$$=\left(\begin{array}{c}
			0 \\ APv+\varepsilon\Lambda_{\varepsilon t}^{\theta}(\varepsilon\xi+A\eta)
			\end{array}\right)=\left(\begin{array}{c}
			0 \\ A(Pv+\varepsilon\Lambda_{\varepsilon t}^{\theta}\eta)+\Lambda_{\varepsilon t}^{\theta}\xi
			\end{array}\right)=\XC(\varepsilon t, Pv+\varepsilon\Lambda_{\varepsilon t}^{\theta}\eta)=\XC(\psi(t, v)),$$
			where in the third equality we used that 
			$$P\theta=\varepsilon\theta P\;\;\;\;\implies\;\;\;\;P\rme^{t\theta}=\rme^{\varepsilon t\theta}P\;\;\;\;\mbox{ and }\;\;\;\;P\Lambda_t^{\theta}P^{-1}=\varepsilon\Lambda_{t\varepsilon}^{\theta},$$
			and in the fourth equality that $A\theta=\theta A$, showing that $\XC_{\psi}$ and $\XC$ are $\psi$-conjugated. Since automorphisms preserves left-invariant vector fields, it holds that $\Sigma_{\psi}$ is in fact an ARS on $G(\theta)$. Moreover, if we define the left-invariant metric on $\Delta_{\psi}$ that makes $(d\psi)_{(0, 0)}\bigl|_{\Delta_{\psi}}$ an isometry, we get that $(d\psi)_{(t, v)}$ carries orthonormal frames in $\Delta_{\psi}(t, v)$ onto orthonormal frames in $\Delta(\psi(t, v))$ implying that $\psi$ is in fact an isometry between $\Sigma_{\psi}$ and $\Sigma.$
		\end{proof}
		
		\bigskip
		\begin{remark}
			\label{simplification}
   Let $\Sigma=\{\XC=(\xi, A), \Delta^L\}$ be an ARS on $G(\theta)$. The maps
	   $$\psi_1(t, v)=\left(t, v-\Lambda_t^{\theta}(A^{-1}\xi_1)\right), \hspace{.5cm}\mbox{ if }\hspace{.5cm}\det A\neq 0,$$
	   and 
	   $$\psi_2(t, v)=\left(t, v-\frac{1}{\sigma}\Lambda_t^{\theta}u\right), \hspace{.5cm}\mbox{ where }\hspace{.5cm}(\sigma, u)\in\Delta\;\mbox{ with }\;\sigma\neq 0,$$
	   are automorphisms of $G(\theta)$ and their induced ARS's satisfies $\XC_{\psi_1}=(0, A)$ and $(1, 0)\in\Delta_{\psi_2}$. Therefore, up to automorphisms we can assume that $(1, 0)\in\Delta$ {\bf or} $\xi=0$ if $\det A=0$.	   
		\end{remark}

        \bigskip
        
        Next we analyze some properties of the singular locus of ARS's on the groups $G(\theta)$.

		\begin{proposition}
			\label{LARC}
			For the ARS $\Sigma=\{\XC=(\xi, A), \Delta^L\}$ on $G(\theta)$, it holds:
			\begin{enumerate}
				\item $\Delta$ is a subalgebra if and only if $l_{\Delta}$ is an eigenspace of $\theta$;
				\item If $(1, 0)\in\Delta$, then $\Sigma$ satisfies the LARC if and only if $\Delta$ is not a subalgebra or $\Delta$ is a subalgebra and $Al_{\Delta}\not\subset l_{\Delta}$ or $\xi\notin l_{\Delta}$;
			\end{enumerate}	
		\end{proposition}
		
		\begin{proof}
			1. Let $\{(\sigma_1, u_1), (\sigma_2, u_2)\}$ be a basis of $\Delta$. Then,
			$$\Delta\ni-\sigma_2(\sigma_1, u_1)+\sigma_1(\sigma_2, u_2)=(0, -\sigma_2u_1+\sigma_1 u_2)\in\{0\}\times\R^2,$$
			is a nonzero vector, implying that $l_{\Delta}=\R (\sigma_1 u_2-\sigma_2u_1)$.
			On the other hand,  		 
			$$\Delta\;\;\mbox{ is a subalgebra }\;\;\iff\;\; [(\sigma_1, u_1), (\sigma_2, u_2)]\in\Delta.$$
			However, by definition
			$$[(\sigma_1, u_1), (\sigma_2, u_2)]=(0, \sigma_1\theta u_2-\sigma_2\theta u_1)=(0, \theta(\sigma_1 u_2-\sigma_2u_1)),$$
			implying that 
			$$\Delta\;\;\mbox{ is a subalgebra }\;\;\iff\;\; l_{\Delta}\;\;\mbox{ is an eigenspace of }\theta.$$
			
			2. If Let $\{(\sigma_1, u_1), (\sigma_2, u_2)\}$ be an orthonormal basis of $\Delta$. If $\Delta$ is not a subalgebra, then necessarily $[(\sigma_1, u_1), (\sigma_2, u_2)]\notin\Delta$ and the LARC is satisfied. On the other hand, if $\Delta$ is a subalgebra, then by the previous lemma, $l_{\Delta}$ is an eigenspace of $\theta$. Since by hypothesis $(1, 0)$ also belongs to $\Delta$ we have that $\Delta=\R(1, 0)\oplus l_{\Delta}$. Therefore, if $\DC$ is the derivation associated with $\XC$, the LARC is satisfied if and only
			$$\DC(\{0\}\times l_{\Delta})\not\subset\Delta\;\;\;\mbox{ or }\;\;\;\DC(1, 0)\notin\Delta.$$
			However, the fact that 		
			$$\DC=\left(\begin{array}{cc}
			0 & 0 \\ \xi & A
			\end{array}\right)\;\;\mbox{ gives us that }\;\;\DC(\{0\}\times l_{\Delta})=\{0\}\times Al_{\Delta}\;\;\mbox{ and }\;\;\DC(1 ,0)=(0, \xi),$$
			showing that $\DC\Delta\not\subset\Delta$ if and only if $Al_{\Delta}\not\subset l_{\Delta}$ or $\xi\notin l_{\Delta}$, concluding the proof.
		\end{proof}

		\begin{remark}
			It is important to notice that if $\theta\neq\id_{\R^2}$, the fact that $A\theta=\theta A$ gives us that
			$$\theta l_{\Delta}\subset l_{\Delta}\;\;\;\implies \;\;\; Al_{\Delta}\subset l_{\Delta}.$$
			Consequently, in this case, if $(1, 0)\in \Delta$ the LARC holds if and only if $\xi\notin l_{\Delta}$.
		\end{remark}

			\section{The singular locus}

			In this section we analyze the singular locus of the ARS's on $G(\theta)$. As showed in \cite[Theorem 1]{PJAyGZ1}, the singular locus of a simple ARS $\Sigma=\{\XC, \Delta^L\}$ such that $\Delta$ is a subalgebra, is a submanifold. In this section we show that for the class of groups in question the same holds independent of any condition on the distribution. Such property is not true, for instance, on the Heisenberg group as showed in \cite[Section 3.5.3]{PJAyGZ1}.

		   	Let $\Sigma=\{\XC=(\xi, A), \Delta^L\}$ be a simple ARS on $G(\theta)$ wich associated singular locus 
		  	$$\ZC:=\{(t, v)\in G(\theta); \;\XC(t, v)\in \Delta(t, v)\}.$$
		  	Using the expression for $\XC$ gives us that 
		  	$$\XC(t, v)=(0, Av+\Lambda^{\theta}_t\xi)\in\Delta^L(t, v)\;\;\iff\;\; (0, Av+\Lambda_t\xi)\in\Delta^L(t, v)\cap(\{0\}\times\R^2)=\{0\}\times \rho_t (l_{\Delta}).$$
		  	Therefore, if ${\bf u}$ is a vector normal to $l_{\Delta}$, it holds that 
		  	$$\ZC=\{(t, v)\in G(\theta); \;\;\;\langle\rho_{-t}(Av+\Lambda_t^{\theta}\xi), {\bf u}\rangle_{\R^2}=0\}.$$
		  	Defining 
		  	\begin{equation}
		  	\label{function}
		  	F_{\bf u}:G(\theta)\rightarrow\R, \hspace{2cm} F_{\bf u}(t, v)=\langle\rho_{-t}(Av+\Lambda_t^{\theta}\xi), {\bf u}\rangle_{\R^2},
		  	\end{equation}
		  	gives us that $\ZC=F^{-1}(0)$. Moreover, a simple calculation shows that
		  	\begin{equation}
		  	\label{eq}
		  \partial_1F_{\bf u}(t, v)=\langle\rho_{-t}\left(\xi-\theta Av\right), {\bf u}\rangle_{\R^2}\hspace{2cm}\mbox{ and }\hspace{2cm}\partial_2F_{\bf u}(t, v)w=\langle\rho_{-t}Aw, {\bf u}\rangle_{\R^2}.
		  	\end{equation}
		  	
		  	\begin{theorem}
		  		The singular locus of any ARS on a $G(\theta)$ is an embedded submanifold.
		  	\end{theorem}
		  
		  \begin{proof}
		  	Let us consider an ARS $\Sigma=\{\XC, \Delta^L\}$. Since the image of a submanifolds by a diffeomorphism is also a submanifold, Proposition \ref{difeo} allows us to assume w.l.o.g. that $(1, 0)\in \Delta$. 
		  	
		  	Assume first that $A\equiv 0$. By definition, in this case, the singular locus is given by 
		  	$$\ZC=X\times \R^2, \;\;\mbox{ where }\;\; X=\{t\in \R;\;\;\; \langle\rho_{-t}(\Lambda_t^{\theta}\xi), {\bf u}\rangle_{\R^2}=0\}.$$
		  	However, 
		  	$$\langle\rho_{-t}(\Lambda_t^{\theta}\xi), {\bf u}\rangle_{\R^2}=-\langle\Lambda_{-t}^{\theta}\xi, {\bf u}\rangle_{\R^2}=\left\{\begin{array}{c}
		  	\xi_1u_1(1-\rme^{-t})+\xi_2u_2t, \;\;\mbox{ if }\;\;\det\theta = 0\\
		  	-\langle\rho_{-t}\theta^{-1}\xi, {\bf u}\rangle_{\R^2}+\langle\theta^{-1}\xi, {\bf u}\rangle_{\R^2}, \;\;\mbox{ if }\;\;\det\theta\neq 0
		  	\end{array}\right.,$$
		  	where $\xi=(\xi_1, \xi_2)$ and ${\bf u}=(u_1, u_2)$ in the canonical basis.
		  	
		  	By the LARC we get that $\xi\notin l_{\Delta}$ and hence, differentiation of the first equation shows that $t\mapsto \xi_1u_1(1-\rme^{-t})+\xi_2u_2t$ has at most one critical point implying by Rolle's Theorem that $X$ has at most two points if $\det\theta=0$. Now, if $\det\theta\neq 0$, we get by Lemma \ref{Lambda}, that $X$ is an enumerable subset of $\R$ with cardinality depending on the eigenvalues of $\theta$. In any case, the singular locus of $\Sigma$ is an enumerable union of planes on $G(\theta)$, and hence a submanifold as stated.

		  	Let us now consider the case where $A\not\equiv 0$. In order to show that the singular locus $\ZC$ of $\Sigma$ is a submanifold, it is enough to assure that $0\in\R$ is a regular value of the map $F_{\bf u}$ defined previously. 
		  	
		  	By equation (\ref{eq}), $0\in\R$ is not a regular value of $F$ if and only if there exists $(t, v)\in\ZC$ such that 
		  	$$\partial_1F_{\bf u}(t, v)=\langle\rho_{-t}\left(\xi-\theta Av\right), {\bf u}\rangle_{\R^2}=0\hspace{1cm}\mbox{ and }\hspace{1cm}\forall w\in\R^2\;\;\;\partial_2F_{\bf u}(t, v)w=\langle\rho_{-t}Aw, {\bf u}\rangle_{\R^2}=0.$$
		  	From the previous, it holds that 
		  	$$\forall w\in\R^2, \;\;\;\partial_2F_{\bf u}(t, v)w=\langle\rho_{-t}Aw, {\bf u}\rangle_{\R^2}=0\;\;\implies\;\;\mathrm{Im}A\subset l_{\Delta},$$
		  	and 
		  	$$0=\langle\rho_{-t}\left(\xi-\theta Av\right), {\bf u}\rangle_{\R^2}=\langle\rho_{-t}\xi, {\bf u}\rangle_{\R^2}-\underbrace{\langle \rho_{-t}A\theta v, {\bf u}\rangle_{\R^2}}_{=0}=\langle\rho_{-t}\xi, {\bf u}\rangle_{\R^2},$$
		  	where we used that $A\theta=\theta A$. Since $A\not\equiv 0$ we have that $\mathrm{Im}A=l_{\Delta}$ implying, in particular, that $Al_{\Delta}\subset l_{\Delta}$ and $\Delta$ is a subalgebra.

		  	On the other hand,   
		  	$$\langle\rho_{-t}\xi, {\bf u}\rangle_{\R^2}=0\;\;\implies\;\;\rho_{-t}\xi\in l_{\Delta}\;\;\implies\;\;\xi\in\rho_{t}(l_{\Delta})=l_{\Delta},$$
		  	which by Proposition \ref{LARC} implies that $\Sigma$ does not satisfies the LARC. Therefore, $0$ is a regular value of the map $F_{\bf u}$ showing that $\ZC$ is in fact an embedded submanifold.		    	
		  \end{proof}
		  
		  \bigskip
		  
		  The next result studies the connectedness of the singular locus.
		  
		  \begin{theorem}
		  	\label{connectedcomponents}
		  	If $\Sigma=\{\XC=(\xi, A), \Delta^L\}$ is an ARS and $A\not\equiv 0$, then $\ZC$ is connected. Moreover, $G(\theta)\setminus\ZC$ has two connected components given by 
		  	$$\CC^-:=F_{\bf u}^{-1}(-\infty, 0)\;\;\mbox{ and }\;\;\CC^+:=F_{\bf u}^{-1}(0, +\infty),$$
		  	where $F_{\bf u}$ is the function define in (\ref{function}).
		  \end{theorem}
		  
		  \begin{proof}
		  	Assume first that $\det A\neq 0$ and consider
		  	$$H:G(\theta)\rightarrow G(\theta), \;\;\; H(t, v)=(t, A^{-1}(\rho_tv-\Lambda^{\theta}_t\xi)).$$
		  	The map $H$ is a homeomorphism with inverse given by 
		  	$$H^{-1}(t, v)=(t, \rho_{-t}(Av+\Lambda^{\theta}_t\xi)).$$
		  	Moreover, 
		  	$$(t, v)\in \ZC\;\iff\; \langle \rho_{-t}(Av+\Lambda^{\theta}_t\xi), {\bf u}\rangle_{\R^2}=0\;\iff \exists s\in\R; \;\rho_{-t}(Av+\Lambda^{\theta}_t\xi)=su$$
		  	$$\iff v=A^{-1}(s\rho_{t}u-\Lambda_t\xi)\;\;\iff (t, v)=H(t, su),$$
		  	showing that $\ZC$ is the homeomorphic image of the plane $\R\times l_{\Delta}$. In particular, $\ZC$ is connected and $G(\theta)\setminus\ZC$ has two connected components. 
		  	Also, 
		  	$$F_{\bf u}(H(t, v))=F_{\bf u}\left(t, A^{-1}(\rho_tv-\Lambda^{\theta}_t\xi)\right)=\langle\rho_{-t}\left(A\left(A^{-1}(\rho_tv-\Lambda^{\theta}_t\xi)\right)+\Lambda^{\theta}_t\xi\right), {\bf u}\rangle_{\R^2}=\langle v, {\bf u}\rangle_{\R^2},$$
		  	implying that $F_{\bf u}^{-1}(-\infty, 0)$ and $F_{\bf u}^{-1}(0, +\infty)$ are (pathwise) connected. Since 
		  	$$F_{\bf u}^{-1}(\R\setminus\{0\})=G(\theta)\setminus\ZC,$$
		  	we get that $\CC^-$ and $\CC^+$ are in fact the connected components of $G(\theta)\setminus\ZC$.
		  	\bigskip
		  	
		  	Let us now consider the case where $\dim \mathrm{Im}A=1$. By Proposition \ref{difeo} (see also Remark \ref{simplification}) we can assume that $(1, 0)\in\Delta$ . 
		  	Since $\ker A$ has also dimension one and $A\theta=\theta A$, we can easily construct an orthonormal basis $\{w_1, w_2\}$ of $\R^2$ such that 
		  	$$Aw_1\neq 0, \;\;Aw_2=0\;\;\;\mbox{ and }\;\;\;\theta Aw_1=\beta Aw_1.$$
		  	Let $(t, v)\in\ZC$ and write $v=\langle v, w_1\rangle_{\R^2}w_1+\langle v, w_2\rangle_{\R^2}w_2$. Then,
		  	$$0=\langle \rho_{-t}(Av+\Lambda_t^{\theta}\xi), {\bf u}\rangle_{\R^2}=\langle v, w_1\rangle_{\R^2}\langle \rho_{-t}Aw_1, {\bf u}\rangle_{\R^2}+\langle\rho_{-t}\Lambda_t^{\theta}\xi, {\bf u}\rangle_{\R^2}\;\;\implies\;\; $$
		  	$$\langle v, w_1\rangle_{\R^2}\langle \rho_{-t}Aw_1, {\bf u}\rangle_{\R^2}=\langle\Lambda_{-t}^{\theta}\xi, {\bf u}\rangle_{\R^2}.$$
		  	Since $Aw_1$ is a nonzero eigenvector of $\theta$ we have that  
		  	$$\langle \rho_{-t}Aw_1, {\bf u}\rangle_{\R^2}=\rme^{-\beta t}\langle Aw_1, {\bf u}\rangle_{\R^2},$$
		  	where $\beta$ is an eigenvalue of $\theta$. In particular, if $\langle Aw_1, {\bf u}\rangle_{\R^2}=0$ we get by orthogonality that $Aw_1\in l_{\Delta}$, implying that $l_{\Delta}=\mathrm{Im}A$ and hence $Al_{\Delta}\subset l_{\Delta}$. Also, in this case,  
		  		$$\forall t\in\R, \;\;\; 0=\rme^{-\beta t}\langle Aw_1, {\bf u}\rangle_{\R^2}=\langle \rho_{-t}Aw_1, {\bf u}\rangle_{\R^2}=\langle\Lambda_{-t}^{\theta}\xi, {\bf u}\rangle_{\R^2}\;\;\implies\;\;\forall t\in\R, \;\;\Lambda_{-t}^{\theta}\xi\in l_{\Delta},$$
		  		and since $\det \Lambda_{-t}^{\theta}\neq 0$ for some $t\in\R$ we get by the properties of the operator $\Lambda^{\theta}$ that $\xi\in l_{\Delta}$. Therefore, 
		  		$$\langle Aw_1, {\bf u}\rangle_{\R^2}=0\;\;\implies\;\; Al_{\Delta}\subset l_{\Delta}\;\mbox{ and }\;\xi\in l_{\Delta},$$
		  		which together with the assumption  $(1, 0)\in\Delta$ contradicts the LARC (see Proposition \ref{LARC}). Therefore, $\langle Aw_1, {\bf u}\rangle_{\R^2}\neq 0$ and consequently 
		  		$$\langle v, w_1\rangle_{\R^2}=\frac{\langle \Lambda_{-t}^{\theta}\xi, {\bf u}\rangle_{\R^2}}{\langle \rho_{-t}Aw_1, {\bf u}\rangle_{\R^2}}.$$
		  		By the previous, the map
		  		$$I:G(\theta)\rightarrow G(\theta), \;\;(t, v)\mapsto I(t, v)=\left(t, \left(\langle v, w_1\rangle_{\R^2}+\frac{\langle \Lambda_{-t}^{\theta}\xi, {\bf u}\rangle_{\R^2}}{\langle \rho_{-t}Aw_1, {\bf u}\rangle_{\R^2}}\right)w_1+ \langle v, w_2\rangle_{\R^2}w_2\right),$$
		  		is well defined, continuous and a simple calculation shows that its inverse in given by 
		  		$$I^{-1}(t, v)=\left(t, \left(\langle v, w_1\rangle_{\R^2}-\frac{\langle \Lambda_{-t}^{\theta}\xi, {\bf u}\rangle_{\R^2}}{\langle \rho_{-t}Aw_1, {\bf u}\rangle_{\R^2}}\right)w_1+ \langle v, w_2\rangle_{\R^2}w_2\right).$$
		  		Moreover, by the previous calculations, $I(\R\times\R w_2)=\ZC$, showing that $\ZC$ is connected. Also, 
		  		$$F_{\bf u}(I(t, v))=\left\langle \rho_{-t}A\left(\left(\langle v, w_1\rangle_{\R^2}+\frac{\langle \Lambda_{-t}^{\theta}\xi, {\bf u}\rangle_{\R^2}}{\langle\rho_{-t} Aw_1, {\bf u}\rangle_{\R^2}}\right)w_1+ \langle v, w_2\rangle_{\R^2}w_2\right)+\rho_{-t}\Lambda_t^{\theta}\xi, {\bf u}\right\rangle_{\R^2}$$
		  		$$=\left(\langle v, w_1\rangle_{\R^2}+\frac{\langle\Lambda_{-t}^{\theta}\xi, {\bf u}\rangle_{\R^2}}{\langle\rho_{-t} Aw_1, {\bf u}\rangle_{\R^2}}\right)\langle \rho_{-t}Aw_1, {\bf u}\rangle_{\R^2}+\langle\rho_{-t}\Lambda_t^{\theta}\xi, {\bf u}\rangle_{\R^2}$$
		  		$$=\langle v, w_1\rangle_{\R^2}\langle \rho_{-t}Aw_1, {\bf u}\rangle_{\R^2}+\langle\Lambda_{-t}^{\theta}\xi, {\bf u}\rangle_{\R^2}-\langle \Lambda_{-t}^{\theta}\xi, {\bf u}\rangle_{\R^2}$$
		  		$$=\rme^{-\beta t}\langle v, w_1\rangle_{\R^2}\langle Aw_1, {\bf u}\rangle_{\R^2},$$
		  		which as previously implies that $\CC^-$ and $\CC^+$ are the connected components of $G(\theta)\setminus\ZC$.
		  \end{proof}

	The next example shows that if the matrix $A$ associated with a linear vector field $\XC=(\xi, A)$ is trivial, the singular locus is not necessarily connected.
	
	\begin{example}
		\label{example1}
		Consider the 3D Lie group $G(\theta)$ and the simple ARS $\Sigma=\{\XC, \Delta^L\}$, where
		$$\theta=\begin{pmatrix}
			0 & -1   \\
			1 & 0   
		\end{pmatrix},\;\;\;\;\;\XC=(\xi, 0)\;\;\;\;\mbox{ and }\;\;\;\;\alpha=\{(1,0),(0,e_{1})\},$$
	    is an orthonormal basis of $\Delta$. Since $[(1,0),(0,e_1)]=(0,e_{2})$, the subspace $\Delta$ is not a subalgebra and consequently $\Sigma$ satisfies the LARC. 
	  Moreover, the fact that       	
		$$\Lambda_t^{\theta}=(\rho_t-\operatorname{id}_{\R^{2}})\theta^{-1}=\begin{pmatrix}
			\sin t & \cos t-1   \\
			1-\cos t & \sin t   
		\end{pmatrix},$$
       gives us that, in coordinates,
	   $$\XC(t,(x, y))=(0,a \sin t +b(\cos t -1),a(1-\cos t)+b\sin t), \;\;\;\mbox{ where }\;\;\;\xi=(a, b).$$
	   Therefore, we obtain 
	 	$$
			\ZC=\{(t, (x, y)); \;\; b\sin t+a\cos t =a\}.
		$$
	    We claim that $\ZC$ has an infinite number of connected components. In fact, if $\gamma$ is the angle between the vectors $\xi$ and $e_1$ we can write 
		$$\cos \gamma=\dfrac{a}{a^2+b^2}\;\;\;\mbox{ and }\;\;\;\sin \gamma=\dfrac{b}{a^2+b^2},$$
		and hence,  
		$$b\sin t+a\cos t =a\;\;\;\iff\;\;\;\cos\gamma=\cos(t-\gamma).$$
		As a consequence, the singular locus $\ZC$ is given by 
		\begin{equation}
			\label{loci1}
		\ZC=X\times\R^2, \;\;\;\mbox{ where }\;\;\; \Gamma=\{2(\gamma-\pi k)\;k \in \Z \}\cup \{2\pi k\;k \in \Z \},
	\end{equation}
		showing it has an infinite enumerable union of parallel planes (see Figure \ref{Locus1}).

	\end{example}

	\begin{example}
			\label{example2}
		Consider now $G(\theta)$ and the simple ARS $\Sigma=\{\XC, \Delta^L\}$, where
		$$\theta=\begin{pmatrix}
			0 & -1   \\
			1 & 0   
		\end{pmatrix},\;\;\;\;\;\XC=\left((1, 3), \left(\begin{array}{cc}
		2 & 1 \\  0 & 2
	\end{array}\right)\right)\;\;\;\;\mbox{ and }\;\;\;\;\alpha=\{(1,0),(0,e_{1})\},$$
		is an orthonormal basis of $\Delta$. In this case, $[(1,0),(0,e_1)]=(0,e_1)$, implying that $\Delta$ is a subalgebra. However, $(1, 3)\notin\Delta$ implying that the LARC is satisfied. By straightforward calculations, we get that, in coordinates, 
		$$\XC(t, (x, y))=(0,2x+y+3t\rme^{t}-2\rme^t+2, 2y+3(\rme^{t}-1)),$$
		and hence,
		$$\ZC=\{(t, (x, y)); \;\; 2y=3(1-\rme^t)\}.$$
	\end{example}
\begin{figure}[h]
	\centering
	\begin{subfigure}{.5\textwidth}
		\centering
		\includegraphics[width=1\linewidth]{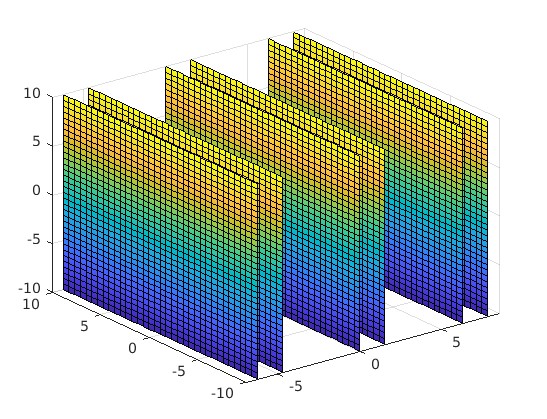}
		\caption{The singular locus of Example \ref{example1}}
		\label{Locus1}
	\end{subfigure}%
	\begin{subfigure}{.5\textwidth}
		\centering
		\includegraphics[width=1\linewidth]{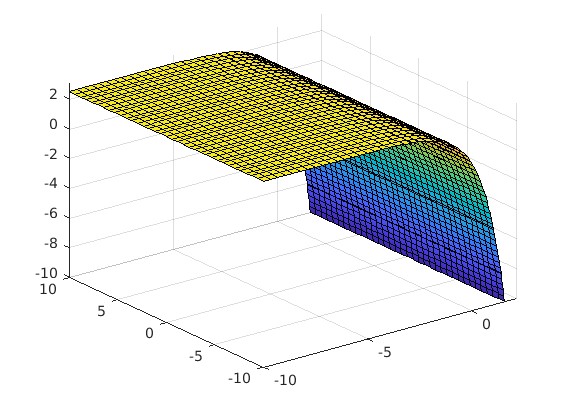}
		\caption{The singular locus of Example \ref{example2}}
		\label{Locus2}
	\end{subfigure}
\end{figure}

    \subsection{Crossing the singular locus}
		  
		  In this section we analyze how the exponential curves crosses the singular locus.
		  
		  Let us fix a simple ARS $\Sigma=\{\XC=(\xi, A), \Delta^L\}$ with $A\neq 0$ and consider, as previously, the function $F_{\bf u}$, where ${\bf u}$ is a normal vector to the line $l_{\Delta}$. Let $(t, v)\in G(\theta)$ and consider the exponential curve $s\in\R\mapsto (t, v)\exp s(a, w)$. By Theorem \ref{connectedcomponents} in order to see how such curve behaves with relation to the singular locus, it is enough to analyze the zeros of the function 
		  $$s\in\R\mapsto F_{\bf u}((t, v)\exp s(a, w)).$$
		  Using the formula for the exponential and for the operator $\Lambda^{\theta}$ is easy to see that 
		  \begin{equation}
		  \label{zeros}
		  F_{\bf u}((t, v)\exp s(a, w))= \left\{\begin{array}{ll}
		  F_{\bf u}(t, v)+s\langle Aw, {\bf u}\rangle_{\R^2} & \mbox{ if } a=0\\
		  F_{\rho^T_{-as} \bf u}(t, v)-\frac{1}{a}\langle \Lambda_{-as}(Aw+a\xi), {\bf u}\rangle_{\R^2} & \mbox{ if } a\neq 0\\
		  \end{array}\right.,
		  \end{equation}
		  where for simplicity we $F_{\rho^T_{-as} \bf u}(t, v)=\langle \rho_{-t}(Av+\Lambda_t\xi), \rho^T_{-as}{\bf u}\rangle_{\R^2}=\langle \rho_{-t-as}(Av+\Lambda_t\xi), {\bf u}\rangle_{\R^2}$.
		  
		  The following lemma states what happens with the previous function when $(t, v)\in\ZC$. 
		  
		  \begin{lemma}
		  	Let $(t, v)\in\ZC$ and consider the function 
		  	$$s\mapsto F_{\bf u}((t, v)\exp s(a, w)).$$
		  	Then, $F_{\bf u}((t, v)\exp s(a, w))\equiv 0$ or there exists $\delta>0$ such that $F_{\bf u}((t, v)\exp s(a, w))\neq 0$ for all $s\in(-\delta, \delta)\setminus\{0\}$.
		  \end{lemma}
		  
		  \begin{proof}
		  	Let us first consider the case where $a=0$. By equation (\ref{zeros}), if $(t, v)\in\ZC$, 
		  	$$F_{\bf u}((t, v)\exp s(a, w))=\underbrace{F_{\bf u}(t, v)}_{=0}+s\langle Aw, {\bf u}\rangle_{\R^2}=s\langle Aw, {\bf  u}\rangle_{\R^2}.$$
		  	Therefore, if $Aw\in l_{\Delta}$ we get that $F_{\bf u}((t, v)\exp s(a, w))=0$ for all $s\in\R$ and if $Aw\notin l_{\Delta}$ we get that 
		  	$F_{\bf u}((t, v)\exp s(a, w)\neq 0$ for all $s\in \R\setminus\{0\}$.
		  	
		  	Assume now that $a\neq 0$ and fix $0\neq u\in l_{\Delta}$. By equation (\ref{zeros}) we have that 
		  	$$F_{\bf u}((t, v)\exp s(a, w))=F_{\rho^T_{-as} \bf u}(t, v)-\frac{1}{a}\langle \Lambda_{-as}(Aw+a\xi), {\bf u}\rangle_{\R^2}.$$
		    Since $(t, v)\in \ZC$, there exists $\mu=\mu(t, v)\in\R$ such that $\rho_{-t}(Av+\Lambda_t\xi)=\mu u$. In particular, 
		    $$F_{\rho^T_{-as} \bf u}(t, v)=\langle \rho_{-t}(Av+\Lambda_t\xi), \rho^T_{-as}{\bf u}\rangle_{\R^2}=\mu\langle u, \rho^T_{-as}{\bf u}\rangle_{\R^2}=\mu\langle\rho_{-as}u, {\bf u}\rangle_{\R^2},$$
		  	and so, 
		   $$F_{\bf u}((t, v)\exp s(a, w))=\mu\langle\rho_{-as}u, {\bf u}\rangle_{\R^2}-\frac{1}{a}\langle\Lambda_{-as}(Aw+a\xi), {\bf u}\rangle_{\R^2}.$$
		   We have the following cases:
		   \begin{enumerate}
		   	\item $Aw+a\xi-a\mu\theta u\in l_{\Delta}$. 
		   	
		   	In this case, there exists $\tau\in\R$ such that $Aw+a\xi-a\mu\theta u=\tau u$ and hence 
		   	$$\mu\rho_{-as}u-\frac{1}{a}\Lambda^{\theta}_{-as}(Aw+a\xi)=\mu(\rho_{-as}u-\Lambda^{\theta}_{-as}\theta u)-\frac{\tau}{a}\Lambda^{\theta}_{-as}u=\mu u-\frac{\tau}{a}\Lambda^{\theta}_{-as}u,$$
		   	where for the last equality we used property 4. of the operator $\Lambda^{\theta}$. Hence, 
		   	$$F_{\bf u}((t, v)\exp s(a, w))=\mu\langle\rho_{-as}u, {\bf u}\rangle_{\R^2}-\frac{1}{a}\langle\Lambda_{-as}(Aw+a\xi), {\bf u}\rangle_{\R^2}$$
		   	$$=\mu \langle u, {\bf u}\rangle_{\R^2}-\frac{\tau}{a}\langle\Lambda^{\theta}_{-as}u, {\bf u}\rangle_{\R^2}=-\frac{\tau}{a}\langle\Lambda^{\theta}_{-as}u, {\bf u}\rangle_{\R^2}.$$
		   	Therefore, if $\tau=0$ or $\Delta$ is a subalgebra, $-\frac{\tau}{a}\langle\Lambda^{\theta}_{-as}u, {\bf u}\rangle_{\R^2}=0$, implying that $F_{\bf u}((t, v)\exp s(a, w))\equiv 0$. Now, if $\tau\neq 0$ and $\Delta$ is not a subalgebra, we have that 
		   	$$\frac{d}{ds}F_{\bf u}((t, v)\exp s(a, w))=\tau\langle\rho_{-as}u, {\bf u}\rangle_{\R^2}.$$
		   	In particular, since $\Delta$ is not a subalgebra, $u$ is not an eigenvalue of $\theta$ and hence, there exists $\delta>0$ such that 
		   	$$\frac{d}{ds}F_{\bf u}((t, v)\exp s(a, w))=\tau\langle\rho_{-as}u, {\bf u}\rangle_{\R^2}\neq 0, \;\;s\in(-\delta, \delta)\setminus\{0\},$$
		   	showing that $0\in\R$ is an isolated critical point. In particular, $F_{\bf u}((t, v)\exp s(a, w))\neq 0$ in $s\in(-\delta, \delta)\setminus\{0\}$ as desired. 
		   		   	
		   	\item $Aw+a\xi-a\mu\theta u\notin l_{\Delta}$
		   		   	
		   	In this case, 
		   	$$\frac{d}{ds}F_{\bf u}((t, v)\exp s(a, w))=\frac{1}{a}\left\langle \rho_{-as}\left(Aw+a\xi-a\mu\theta u\right), {\bf u}\right\rangle_{\R^2}.$$
		   	Since 
		   	$$\frac{d}{ds}_{|s=0}F_{\bf u}((t, v)\exp s(a, w))=\frac{1}{a}\left\langle Aw+a\xi-a\mu\theta u, {\bf u}\right\rangle_{\R^2}\neq 0,$$
		   	we get by continuity that there exists $\delta>$ such that 
		   	$$\frac{d}{ds}F_{\bf u}((t, v)\exp s(a, w))\neq 0, \;\;\forall s\in(-\delta, \delta),$$
		   	and hence $F_{\bf u}((t, v)\exp s(a, w))$ is strictly increasing or strictly decreasing the interval $(-\delta, \delta)$. In particular, 
		   	$$F_{\bf u}((t, v)\exp s(a, w))\neq 0, \;\;\mbox{ for all } s\in(-\delta, \delta)\setminus\{0\},$$
		   	proving the result.	   	
		   \end{enumerate}
		    \end{proof}

		  Using the previous lemma we have the following.

	    \begin{theorem}
	    	\label{curves}
		  	Let $\Sigma=\{\XC=(\xi, A), \Delta^L\}$ be a simple ARS on $G(\theta)$ with $A\neq 0$ and $(a, w)\in\fg(\theta)$. If $(t, v)\in G(\theta)\setminus\ZC$ the exponential curve $s\mapsto (t, v)\exp s(a, w)$ satisfies: 
		  	\begin{enumerate}
		  		\item $s\mapsto (t, v)\exp s(a, w)$ remains in the same component that contains $(t, v)$ or
		  		
		  		\item $s\mapsto (t, v)\exp s(a, w)$ intersects $\ZC$ discretely. 	  		
		  \end{enumerate}
		  \end{theorem}
		  
		  \begin{proof}
		  Let us assume that for $s_0\in\R$ it holds that $(t, v)\exp s_0(a, w)\in\ZC$. Since  
		  $$(t, v)\exp (s_0+s)(a, w)=\underbrace{((t, v)\exp s_0(a, w))}_{\in\ZC}\exp s(a, w),$$
		  the previous lemma implies the following:
		  
		  \begin{enumerate}
		  	\item For all $s\in\R$ it holds that $F_{\bf u}((t, v)\exp (s_0+s)(a, w))\equiv 0$;
		  	
		  	If this condition holds, we would have that $(t, v)\exp (s_0+s)(a, w)\in\ZC, \forall s\in\R$ and consequently, $(t, v)\in G(\theta)\setminus\ZC\cap\ZC=\emptyset$ which is not possible. Therefore, $s\mapsto(t, v)\exp s(a, w)$ remains in the same component that contains $(t, v)$.
		  	
		  	\item There exists $\delta>0$ such that  $F_{\bf u}((t, v)\exp (s_0+s)(a, w))\neq 0$ for all $s\in (-\delta, \delta)\setminus\{0\}$.
		  	
		  	In this case, $(t, v)\exp (s_0+s)(a, w)$ intersects $\ZC$ at the point $(t, v)\exp s_0(a, w)$ and 
		  	$(t, v)\exp (s_0+s)(a, w)$ remains in the same component if $F_{\bf u}((t, v)\exp (s_0+s)(a, w))\neq 0$ does not changes sign for all $s\in (-\delta, \delta)\setminus\{0\}$ or 
		  	$\{(t, v)\exp (s_0+s)(a, w), \;\;s\in(-\delta, 0)\}$ and $\{(t, v)\exp (s_0+s)(a, w), \;\;s\in(0, \delta)\}$ belongs to different components if the sign of $F_{\bf u}((t, v)\exp (s_0+s)(a, w))\neq 0$ changes in $(-\delta, \delta)\setminus\{0\}$.
		  \end{enumerate}
		  
		  \end{proof}

		  \begin{remark}
		  	\label{crosses}
		    By a simple calculation, it is easy to verify that
		  	$$F_{\bf u}(\varphi_s(t, v))=\langle\rme^{sA}\left(Av+\Lambda_t^{\theta}\xi\right), {\bf u}\rangle_{\R^2}.$$
		  	Consequently, for a fixed $(t, v)\in G(\theta)\setminus\ZC_{\XC}$ the values of $s\in\R$ such that $\varphi_s(t, v)\in \ZC$ are zeros of the function
		  	$$s\in\R\mapsto \langle\rme^{sA}\left(Av+\Lambda_t^{\theta}\xi\right), {\bf u}\rangle_{\R^2},$$
		  and we conclude that
		    \begin{itemize}
		    	\item[(i)] The line $l_{\Delta}$ is an eigenspace of $A$ and 
		    	$$\{\varphi_s(t, v), s\in\R\}\subset\ZC\;\;\iff\;\;\varphi_s(t, v)\in\ZC, \;\mbox{ for some }\;\;s\in\R;$$
		    	
		    	\item[(ii)] The set $\{s\in\R; \;\;\varphi_s(t, v)\in\ZC\}$ is discrete,
		    \end{itemize}
		    showing that the solutions of the linear vector field $\XC$ of an ARS $\Sigma$ crosses its locus discretely or remains inside it.
		    
		  \end{remark}

	\appendix
	
	\section{Asymptotics of $2\times 2$ matrices}

	Let $a, b, c, \theta, \lambda, \lambda_1, \lambda_2\in\R$, with $\lambda_1\neq0\neq\lambda_2$ and consider the functions 
	\begin{equation}
	\label{functions}
	\gamma_1(t)=a\rme^{t\lambda_1}+b\rme^{t\lambda_2}+c, \;\;\;\gamma_2(t)=\rme^{\lambda t}(at+b)+c\;\;\;\mbox{ and }\;\;\;\chi(t)=\rme^{t\lambda}\cos(t+\theta)+c.
	\end{equation}
	
	It is straightforward to see that the following properties holds:
	
	\begin{enumerate}
		\item $a=b=0$ and $\gamma_i\equiv c$;
		\item $ab=0$ with $a^2+b^2\neq 0$ then $\gamma_i$ is unbounded and has at most one zero;
		\item $ab\neq 0$ then $\gamma$ is unbounded and has at most two zeros;
		\item $\chi$ has a finite number of zeros if $|c|>1$;
		\item The zeros of $\chi$ form an infinite discrete subset of $\R$ if $|c|\leq 1$. 
	\end{enumerate}

	\begin{lemma}
		\label{Lambda}
		Let $A\in\mathfrak{gl}(2, \R)$, $\tau\in\R$ and $u, v\in\R^2$ nonzero vectors. For the function 
		$$\gamma:\R\rightarrow\R, \;\;\;\;\gamma(t)=\rme^{tA}u\boldsymbol{\cdot}v+\tau,$$
		it holds:
		\begin{itemize}
			\item[1.] The eigenvalues of $A$ are real and  
			\subitem 1.1. $\gamma$ is an unbounded function with at most two zeros or 
			\subitem 1.2. $\gamma\equiv\tau$, $u\boldsymbol{\cdot}v=0$ and $u$ is an eigenvalue of $A$. 
			
			\item[2.] The eigenvalues of $A$ are complex and
			\subitem 2.1. $\gamma$ is unbounded with an enumerable discrete set of zeros if $\tr A\neq 0$ or
			\subitem 2.2. $\gamma$ is bounded with an enumerable discrete set of zeros if $\tr A=0$.
		\end{itemize}
	\end{lemma}

	\begin{proof}
	    By considering an appropriated orthonormal basis $\alpha$ of $\R^2$ and writing $[u]_{\alpha}=(a, b)$ and $[v]_{\alpha}=(c, d)$ it holds that
		$$\gamma(t)=ac\rme^{t\lambda_1}+bd\rme^{t\lambda_2}+\tau, \;\;\;\mbox{ if }\;\;\;[A]_{\alpha}=\left(\begin{array}{cc}
		\lambda_1 & 0\\0 & \lambda_2
		\end{array}\right),$$
		$$\gamma(t)=\rme^{\lambda t}\left(u\boldsymbol{\cdot} v+\varepsilon bct\right)+\tau\;\;\;\mbox{ if }\;\;\;[A]_{\alpha}=\left(\begin{array}{cc}
		\lambda & \varepsilon\\0 & \lambda
		\end{array}\right),$$
		or
		$$\gamma(t)=\rme^{\lambda t}\|u\|\|v\|\cos(\mu t+\theta_0)+\tau,\;\;\;\mbox{ if }\;\;\;[A]_{\alpha}=\left(\begin{array}{cc}
		\lambda & -\mu\\\mu & \lambda
		\end{array}\right),$$
		where $\theta_0$ is the angle between $u$ and $v$. In particular, the assertions follows from the analysis of the maps $\gamma_1, \gamma_2$ and $\chi$ defined in (\ref{functions}).		
	\end{proof}

		\section{Lifting ARS's to simply connected groups}
	
	In this section we prove some results relating simple ARS on general connected groups and their simply connected coverings. 
	
	Let $G$ be a connected $n$-dimensional Lie group with Lie algebra $\fg$. A vector field $\XC$ on $G$ is said to be {\bf linear} if it is complete and its associated flow $\{\varphi_s\}_{s\in\R}$ is a $1$-parameter subgroup of automorphisms of $G$. Any linear vector field is induces a derivation $\DC\in\mathrm{Der}(\fg)$ defined by 
	$$\DC Y=-[\XC, Y], \;\;\;Y\in\fg\;\;\mbox{ and satisfying }\;\;(d\varphi_s)_e=\rme^{s\DC}, \;\;s\in\R.$$

	A simple ARS $\Sigma$ on $G$ is a pair $\Sigma=\{\XC, \Delta^L\}$, where $\XC$ is a linear vector field and $\Delta^L$ is a left-invariant distribution, endowed with a left-invariant Riemannian metric, satisfying:
	\begin{itemize}
		\item[(i)] The set $\{\XC(g)\in\Delta^L(g)\}$ is nonempty;
		\item[(ii)] $\Sigma$ satisfies the LARC, that is, $\DC\Delta\not\subset\Delta$ or $[\Delta, \Delta]\not\subset\Delta,$ where $\Delta=\Delta^L(e)$.
	\end{itemize}

As in the 3D case, the metric on $G$ associated with $\Sigma$ is defined by declaring that $\XC$ is unitary and orthogonal to $\Delta^L$ at every point of $G$. The singular set $\ZC$ is the set where the metric fail to be Riemannian and is given as
$$\ZC=\{g\in G; \;\XC(g)\in\Delta^L(g)\}.$$
The almost-Riemannian norm is defined as 
$$\|X\|_{\Sigma, g}=\min\left\{\sqrt{\sum_{j=0}^{n-1}\alpha_i^2}; \;\;\alpha_0\XC(g)+\alpha_1Y_1^L(g)+\cdots+\alpha_{n-1}Y^L_{n-1}(g)=X\right\},$$ 
	where $\{Y_1, \ldots, Y_{n-1}\}$ is an orthonormal basis of $\Delta$. If $\Sigma_1$ and $\Sigma_2$ are simple ARS's on $G$ we denote by $\mathrm{Iso}_G(\Sigma_1; \Sigma_2)$ the group of isometries between $\Sigma_1$ and $\Sigma_2$, that is, $\psi\in\mathrm{Iso}_G(\Sigma_1; \Sigma_2)$ if and only if $\psi:G\rightarrow G$ is a diffeomorphism satisfying 
	$$\forall g\in G, v\in T_gG; \;\;\;\;\;\|(d\psi)_gv\|_{\Sigma_2, \psi(g)}=\|v\|_{\Sigma_1, g}.$$
	We also define the set of isometries fixing the identity as 
	$$\mathrm{Iso}_G(\Sigma_1; \Sigma_2)_0=\left\{\psi\in\mathrm{Iso}_G(\Sigma_1; \Sigma_2); \psi(e)=e\right\}.$$

	Let us denote by $\widetilde{G}$ the connected simply connected covering of $G$ and by $\pi:\widetilde{G}\rightarrow G$ the canonical homomorphism. For a simple ARS $\Sigma=\{\XC, \Delta^L\}$ on $G$ we define the {\bf lift} of $\Sigma$ to be the unique simple ARS $\widetilde{\Sigma}=\{\widetilde{\XC}, \widetilde{\Delta}^L\}$ on $\widetilde{G}$ satisfying
    \begin{equation}
    	\label{lift}
    	\pi_*\widetilde{\XC}=\XC\circ\pi\;\;\;\mbox{ and }\;\;\;\pi_*\widetilde{\Delta}^L=\Delta\circ\pi,
    \end{equation}
where the invariant metric in $\widetilde{\Delta}^L$ is the one that turns the restriction $$(d\pi)_g|_{\widetilde{\Delta}^L(g)}:\widetilde{\Delta}(g)\rightarrow\Delta^L(\pi(g)), \;\;\forall \;g\in G,$$
into an isometry.

	The existence of $\widetilde{\XC}$ is guaranteed in by \cite[Theorem 2.2]{AyTi} and the existence of the distribution $\widetilde{\Delta}^L$ comes from the fact that $\pi$ is a homomorphism and a local diffeomorphism. Moreover, the same properties of $\pi$ implies that $\widetilde{\Sigma}$ satisfies conditions (i) and (ii) in the definition of simple ARS's.
	
	Next we show some relations between a simple ARS and its lift. 
	
	\begin{proposition}
		Let $\Sigma$ be a simple ARS and $\widetilde{\Sigma}$ its lift. If $\ZC$ and $\widetilde{\ZC}$ stands for the singular locus of $\Sigma$ and $\widetilde{\ZC}$, respectively, then 
		$$\pi(\widetilde{\ZC})=\ZC\;\;\;\mbox{ and }\;\;\;\pi^{-1}(\widetilde{\ZC})=\ZC.$$
	\end{proposition}

\begin{proof}
	In fact, since $\pi$ is a local diffeomorphism, relation (\ref{lift}) implies that
	$$g\in\widetilde{\ZC}\iff\widetilde{\XC}(g)\in\widetilde{\Delta}^L(g)\iff (d\pi)_g\widetilde{\XC}(g)\in(d\pi)_g\widetilde{\Delta}^L(g)\iff \XC(\pi(g))\in\Delta^L(\pi(g))\iff\pi(g)\in\ZC$$
\end{proof}

Let us notice that the previous result implies that $\ZC$ is a submanifold of $G$ if and only if $\widetilde{\ZC}$ is a submanifold of $\widetilde{G}$. The next result relates isometries of a simple ARS and its lift.

\begin{proposition}
	Let $\Sigma_1, \Sigma_2$ be a simple ARS's on $G$ and $\widetilde{\Sigma}_1$, $\widetilde{\Sigma}_2$ their respective lifts. For any $\psi\in\mathrm{Iso}(\Sigma_1; \Sigma_2)_0$ there exists a unique $\widetilde{\psi}\in \mathrm{Iso}(\widetilde{\Sigma}_1; \widetilde{\Sigma}_2)_0$ satisfying
	\begin{equation}
		\label{liftiso}
			\pi\circ\widetilde{\psi}=\psi\circ\pi.
	\end{equation}
	Moreover, if $\widetilde{\psi}\in\mathrm{Aut}(\widetilde{G})$ then $\psi\in\mathrm{Aut}(G)$.
\end{proposition}

\begin{proof}
	Since $\pi:\widetilde{G}\rightarrow G$ is the simply connected covering of $G$, there exists a unique map $\widetilde{\psi}$ with $\widetilde{\psi}(e)=e$ and satisfying (\ref{liftiso}) and we only have to show that $\widetilde{\psi}\in\mathrm{Iso}(\widetilde{\Sigma}_1; \widetilde{\Sigma}_2)_0$.
	
	Let $\{\varphi^i_s\}_{s\in\R}$ and $\{\widetilde{\varphi}^i_s\}_{s\in\R}$ to be the flows of $\XC_i$ and $\widetilde{\XC}_i$, respectively. For any $s\in\R$ it holds that  
	$$\pi\circ(\widetilde{\psi}\circ\widetilde{\varphi}^1_s)=(\pi\circ\widetilde{\psi})\circ\widetilde{\varphi}^1_s\stackrel{(\ref{liftiso})}{=}(\psi\circ\pi)\circ\widetilde{\varphi}^1_s=\psi\circ(\pi\circ\widetilde{\varphi}_s^1)\stackrel{(\ref{lift})}{=}\psi\circ(\varphi_s^1\circ\pi)=(\psi\circ\varphi_s^1)\circ\pi$$
	$$\stackrel{\psi\in\mathrm{Iso}(\Sigma_1; \Sigma_2)_0}{=}(\varphi_s^2\circ \psi)\circ\pi=\varphi_s^2\circ (\psi\circ\pi)=\stackrel{(\ref{liftiso})}{=}\varphi_s^2\circ (\pi\circ\widetilde{\psi})=(\varphi_s^2\circ \pi)\circ\widetilde{\psi}\stackrel{(\ref{lift})}{=}(\pi\circ\widetilde{\varphi}^2_s)\circ\widetilde{\psi}=\pi\circ(\widetilde{\varphi}^2_s\circ\widetilde{\psi}).$$
	As a consequence, for any $g\in \widetilde{G}$ we have the continuous curve 
	$$s\in\R\mapsto \widetilde{\psi}(\widetilde{\varphi}^1_s(g))\left(\widetilde{\varphi}^2_s(\widetilde{\psi}(g))\right)^{-1}\in\ker\pi.$$
	Since $\ker\pi$ is discrete, we have conclude that 
	$$\widetilde{\psi}\circ\widetilde{\varphi}^1_s=\widetilde{\varphi}^2_s\circ\widetilde{\psi}\;\;\implies\;\;\pi_*\circ\widetilde{\XC}=\XC\circ\pi.$$
	On the other hand, it holds that
	$$(d\pi)_{\widetilde{\psi}(g)}(d\widetilde{\psi})_g\widetilde{\Delta}_1(g)\stackrel{(\ref{liftiso})}{=}(d\psi)_{\pi(g)}(d\pi)_g\widetilde{\Delta}_1(g)\stackrel{(\ref{lift})}{=}(d\psi)_{\pi(g)}\Delta_1(\pi(g))$$
	$$\stackrel{\psi\in\mathrm{Iso}(\Sigma_1; \Sigma_2)_0}{=}\Delta^L_2(\psi(\pi(g)))=\Delta^L_2(\pi(\widetilde{\psi}(g)))=(d\pi)_{\widetilde{\psi}(g)}\widetilde{\Delta}^L_2(\widetilde{\psi}(g)),$$
	and since $\pi$ is a local diffeomorphism, we conclude that 
	$$\widetilde{\psi}_*\widetilde{\Delta}_1=\widetilde{\Delta}^L_2\circ \widetilde{\psi}.$$
	Therefore, in order to show that $\widetilde{\psi}\in \mathrm{Iso}(\widetilde{\Sigma}_1; \widetilde{\Sigma}_2)_0$ it is enough to assure that the restriction 
	$$(d\widetilde{\psi})_g|_{\widetilde{\Delta}_1}:\widetilde{\Delta}_1(g)\rightarrow \widetilde{\Delta}_2(\widetilde{\psi}(g),$$
	is, for all $g\in \widetilde{G}$, an isometry. However, for any $X\in\widetilde{\Delta}_1$, the fact that $\pi$ is an isometry between $\widetilde{\Delta}_i$ and $\Delta_i$, $i=1, 2$, implies that 
	$$\|(d\widetilde{\psi})_gX(g)\|_{\widetilde{\Sigma_2}, \widetilde{\psi}(g)}=\|(d\pi)_{\widetilde{\psi}(g)}(d\widetilde{\psi})_gX(g)\|_{\Sigma_2, \pi(\widetilde{\psi}(g))}\stackrel{(\ref{liftiso})}{=}\|(d\psi)_{\pi(g)}(d\pi)X(g)\|_{\Sigma_2, \psi(\pi(g))}$$
	$$\stackrel{\psi\in\mathrm{Iso}(\Sigma_1; \Sigma_2)_0}{=}\|(d\pi)X(g)\|_{\Sigma_1, \pi(g)}=\|X(g)\|_{\widetilde{\Sigma}_1, g},$$
	which implies that $\widetilde{\psi}\in \mathrm{Iso}(\widetilde{\Sigma}_1; \widetilde{\Sigma}_2)_0$. 

	For the second assertion, let us assume that $\widetilde{\psi}\in\mathrm{Aut}(\widetilde{G})$ and consider $g_1, g_2\in G$. By writting $g_i=\pi(h_i)$, $i=1, 2$ we have that 
	$$\psi(g_1g_2)=\psi(\pi(h_1)\pi(h_2))=\psi(\pi(h_1h_2))\stackrel{(\ref{liftiso})}{=}\pi(\widetilde{\psi}(h_1h_2))$$
	$$\stackrel{\widetilde{\psi}\in\mathrm{Aut}(\widetilde{G})}{=}\pi(\widetilde{\psi}(h_1)\widetilde{\psi}(h_2))=\pi(\widetilde{\psi}(h_1))\pi(\widetilde{\psi}(h_2))\stackrel{(\ref{liftiso})}{=}\psi(\pi(h_1))\psi(\pi(h_2))=\psi(g_1)\psi(g_2),$$
	showing that $\psi\in\mathrm{Aut}(G)$ and concluding the proof.
\end{proof}

\begin{remark}
	Applying the previous results to the 3D solvable nonnilpotent case allow us to conclude that as the main results obtained in Sections 3 and 4 are also true in the case where the group is not necessarely simply connected. 
\end{remark}

\end{document}